\documentclass[12pt,english]{amsart}
\calclayout
\usepackage[off]{auto-pst-pdf}

\usepackage{float}
\usepackage{amsmath,amsfonts,amssymb}
\usepackage{stmaryrd}
\usepackage{graphicx, epsfig, psfrag}
\usepackage{epstopdf}
\usepackage{enumerate}
\usepackage{mathrsfs}
\usepackage{algorithm2e} 
\usepackage{color}
\usepackage{multirow} 
\usepackage{soul}
\usepackage[normalem]{ulem}

\usepackage{lineno,hyperref}
\usepackage[latin1]{inputenc}     % para acentuação em Português
\modulolinenumbers[5]

\newif\ifpdf\ifx\pdfoutput\undefined
\pdffalse
\else
\pdftrue\pdfoutput=1
\fi

\theoremstyle{plain}

\newtheorem{Example}{Example}[section]
\newtheorem{Theorem}{Theorem}[section]
\newcommand\R{\mathbb R}

\def\ba{{\boldsymbol a}}

\def\bG{{\boldsymbol G}}

\definecolor{bluegreen}{rgb}{0,0.75,0.75}

\def\ion{{\text{ion}}}
\def\Na{{\text{Na}}}
\def\K{{\text{K}}}
\def\rand{{\text{rand}}}
\def\Res{{\text{Res}}}
\def\Error{{\text{Error}}}

\def\ext{{\text{ext}}}

\begin{document}

\title[Parameter Identification Problem in the Hodgkin and
Huxley Model]{Parameter Identification Problem in the Hodgkin and
	Huxley Model}

\author[Mandujano Valle, Madureira]{Jemy A. Mandujano Valle, Alexandre L. Madureira}
\address{Departamento de Modelagem Computacional, Laboratório Nacional de Computação Científica, Av. Getúlio Vargas 333, 25651-070 Petrópolis, RJ, Brazil}
\thanks{The second author acknowledges the financial support of the Brazilian funding agency CNPq}

\begin{abstract}

The Hodgkin and Huxley (H-H) model is a nonlinear system of four equations that describes how action potentials in neurons are initiated and propagated, and represents a major advance in the understanding of nerve cells. However, some of the parameters are obtained through a tedious combination of experiments and data tuning. In this paper, we propose the use of an iterative method (Landweber iteration) to estimate some of the parameters in the H-H model, given the membrane electric potential. We provide numerical results showing that the method is able to capture the correct parameters using the measured voltage as data, even in the presence of noise.

\end{abstract}

\date{October 18, 2018}
\maketitle

\section{Introduction.}
In 1952 Hodgkin and Huxley~\cite{H-H1952} used voltage-clamp technique to extract the parameters of the ionic channel model of the squid giant axon.  In the space-clamped version of the H-H model, the membrane electrical potential $V:[0,T]\to\R$ solves   
\begin{equation}\label{equation1}
C_M\dot {V}(t)=I_{\ext}+I_\ion(t)\quad\text{in }(0,T],  
\end{equation}
where $C_M$ is the specific membrane capacitance, $V$ is the membrane potential, $\dot{V}$ is the rate of voltage change (dots denote time derivatives), $I_{\ext}$ is the specific external current applied on the membrane. The specific ionic current  $I_\ion(t)$  is the sum of three currents $(I_\ion(t)=I_{\Na}(t)+I_\K(t)+I_{L}(t))$,
 potassium, sodium and leak currents, satisfying: 
\begin{eqnarray}
I_{\Na}(t)&=&G_{\Na}\;m^a(V,t)\;h(V,t)^b\;(V(t)-E_{\Na});\label{equation2}\\ 
I_\K(t)&=&G_\K\;n^c(V,t)\;(V(t)-E_\K);\label{equation3}\\
I_{L}(t)&=&G_{L}\;(V(t)-E_{L}).\label{equation4}
\end{eqnarray}
The constants $G_{\Na}$, $G_\K$ and $G_L$ are the maximal specific conductance for Na$^+$, K$^+$ and leakage channels, and $E_{\Na}$, $E_\K$, $E_{L}$ are the Nernst equilibrium potentials.
The functions $m$ and $h$ are the activation and inactivation variables for $\Na^+$, and $n$ is the activation function for $K^+$. These functions are unitless gating variables that take values between $0$ and $1$. Also, the exponents $a$, $b$ and $c$ are  positive numbers. The units of the other parameters are in Table \ref{ta1}.

 \begin{table}[ht!]
	{\small	\centering
		\begin{tabular}{|l|l|l|l|l|l|l|l|l|l|}\hline
			\textbf{Parameters }            & \textbf{Units}      & \textbf{Units name}      \\\hline     
			$C_M$                    &$\mu F/cm^2$& microfarad per square centimeter \\\hline	           	
			$V$                    &$m V$       & millivolt   \\\hline
			$\dot V$               &$V/s$       & volts per second   \\\hline
			$I_{\ext}$, $I_{ion}$   &$\mu A/cm^2$& microampere per square centimeter  \\\hline
			$G_{\Na}$, $G_\K$, $G_L$ &$ mS/cm^2$  & millisiemens per square centimeter    \\\hline
			$E_{\Na}$, $E_\K$, $E_L$ &$mV$        & millivolt  \\\hline           
		\end{tabular}
		\caption{Units of the parameters; see \cite{H-H1952}, Table 3. }
		\label{ta1} } 	
\end{table}

The experiments performed by Hodgkin and Huxley \cite{H-H1952} suggest that  $m$, $h$ and $n$ are functions that depend on time and the membrane potential. The exponent $c$ models the number of gating particles on the channel.
 In the case of active $\Na$ currents, experiments suggest that two types of independent gating particles are involved, $a$ activation gates $m$, and $b$ inactivation gates $h$ \cite{ermentrout2003}. In addiction, $m$ $n$ and $h$ satisfy the differential equations:
 \begin{equation}\label{equation5}
   \dot{\mathcal{X}}(V,t)=\alpha_\mathcal{X}(V) (1-\mathcal{X}(V,t))-\beta_\mathcal{X}(V) \mathcal{X}(V,t)
   \quad\text{where }\mathcal{X}=m,n,h.
 \end{equation}
The functions $\alpha_\mathcal{X}$ and $\beta_\mathcal{X}$ depend on the membrane potential and are given by
\begin{equation}\label{e:alphasdef}
\begin{array}{lll}
\alpha_m=\frac{(25-V)/10}{\exp( (25-V)/10)-1},\;\;\; & \alpha_h=0.07\exp(-V/20),\;\;\; & \alpha_\mathsf {n}=\frac{(10-V)/100}{\exp((10-V)/10)-1 },\vspace*{0.2cm}
\\
\beta_m=4\exp(-V/18), & \beta_h=\frac{1}{\exp((30-V)/10)+1},&\beta_\mathsf {n}=0.125\exp(-V/80).
\end{array}
\end{equation}
To equation \eqref{equation1} we add the initial conditions 
\begin{equation}\label{equation6}
V(0)=V_0,\;\;\;m(0)=m_0,\;\;\;n(0)=n_0,\;\;\;h(0)=h_0.
\end{equation}
Thus,~(\ref{equation1}-\ref{equation6}) yield the following system of ordinary differential equation (ODE):
\begin{equation}\label{equation7}
\left \{\begin{array}{l}\displaystyle
C_M\dot {V}=I_{\ext}- G_{\Na}m^ah^b(V-E_{\Na})-G_\K n^c(V-E_\K)-G_{L}(V-E_L)\quad\text{for }t \in (0,T] \vspace*{0.15cm}
\\
\displaystyle \dot{\mathcal{X}} =(1-\mathcal{X})\alpha_\mathcal{X} (V)-\mathcal{X}\beta_\mathcal{X}(V)\quad\text{where }\mathcal{X}=m,n,h\text{ and } t \in (0,T]
\\
V(0)=V_0,\;\;\;\; m(0)=m_0,\;\;\;\; n(0)=n_0,\;\;\;\; h(0)=h_0, 
\end{array} \right.
\end{equation} 
and $C_M$, $I_{\ext}$, $E_{\Na}$, $E_\K$, $E_L$, $m_0$, $n_0$ and $h_0$ are  known.

Given all the parameters, it is possible to find a (theoretical or numerical) solution for~\eqref{equation7}. That is the \emph{direct problem}. In \emph{inverse problems}, one is given the voltage $V$ and has to compute one or more parameters. In this work, we consider two different \emph{inverse} problems. The first one is to obtain the maximum conductances $G_{\Na}$, $G_\K$ and $G_{L}$ given the measurement of the membrane potential. 
For the second problem, the goal is to obtain the exponents $a$, $b$ and $c$, again given the measurement of the membrane potential.
 
Using experimental data from the squid neuron, Hodgkin and Huxley obtained the parameters $a=3$, $b=4$ and $c=1$. Note, however, that other neurons may produce different parameters. 

Besides the Hodgkin and Huxley model, there are simplified models such as the cable equation, FitzHugh-Nagumo and Morris-Lecar models. Wilfrid Rall \cite{rall1977,rall1992-1} developed the use of cable theory in computational neuroscience, as well as passive and active compartmental modeling of the neuron. In a previous paper~\cite{mandujano2018}, the authors determine conductances with nonuniform distribution in the equation of the cable with and without branches, using the Landweber iterative method. See also~\cite{tadi2002,bell2005,avdonin2013,avdonin2015}, for identification of parameters in the cable equation, and~\cite{cox2001-2,cox2004,pavel2012,che2012,pavel2013,tuikina2017} for investigations on inverse problems in FitzHugh-Nagumo and Morris-Lecar models. In \cite{destexhe2007,destexhe2004,vich2017} the authors obtained approximately time-dependent but voltage-independent conductances, given the membrane potential, in a system of three ordinary differential equations (passive membrane equation). For the Hodgkin and Huxley model, the parameters of ionic channels are estimated in~\cite{buhry2011,buhry2012} using evolutionary algorithms.

Inverse problems are said to be \emph{ill-posed}. A problem is ill-posed in the sense of Hadamard~\cite{hadamard2014} if any of the following conditions are not satisfied: there is a solution; the solution is unique; the solution has a continuous dependence on the input data (stability). Here we admit the existence of a single solution to the problem. 
However, stability is not guaranteed. Stability is necessary if we want to ensure that small variations in the data lead to small changes in the solution. Problems of instability can be controlled by regularization methods, in particular the Landweber iterative scheme~\cite{binder1996,chapko2004,hanke1995,neubauer2000}.

This article is outlined as follows. Section~\ref{section2} presents our inverse problems for the H-H model along with some theoretical results, and in Section~\ref{section3} we show numerical results to describe the effectiveness of our strategy. Finally, we include in the Appendices some more technical arguments. 

\section{Inverse Problem  in the H-H model}\label{section2}
In what follows, we describe an abstract formulation of the  Landweber method or Landweber iteration \cite{kaltenbacher2008}. 

Consider~\eqref{equation7} and let  $x=(G_{\Na},G_\K,G_L)\in \mathbb{R}^3$ or $x=(a,b,c)\in \mathbb{R}^3$. Consider also the set of function $L^2(0,T)$, and  the nonlinear operator 
\begin{equation}\label{equa10}
F: \mathbb{R}^3\rightarrow L^2(0,T),
\end{equation}
defined by $F(x)=V$, where $V$ solves~\eqref{equation7}. In practical terms, the data $V$ are obtained by measurements. Therefore, we denote the measurements by $V^\delta$, of the which we assume to know the noise level $\delta$, satisfying
\begin{equation}\label{equa11}
\|V-V^\delta\|_{L^2(0,T)}^2=\int_0^T|V(t)-V^\delta(t)|^2\,dt\le\delta. 
\end{equation}

To obtain an approximation of $x$,  given $V^\delta$,  we used the Landweber iteration
\begin{equation}\label{equation8}
x^{ {k+1},\delta}=x^{ k,\delta}+w^{k,\delta}F'(x^{ k,\delta})^*(V^\delta-F(x^{ k,\delta})), 
\end{equation}
where $F'(x^{k,\delta})$  is the Gateaux-derivative of $F$ computed at $x^{k,\delta}$, and $F'(x^{k,\delta})^*$ is its adjoint. We also define 

\[w^{k,\delta}=\frac{{\|V^\delta-F(x^{k,\delta})\|}^2_{L^2(0,T)}}{{\| F'(x^{ k,\delta})^*(V^\delta-F(x^{ k,\delta}))\|}^2_{\mathbb{R}^3} }.\]

The iteration ~\eqref{equation8} begins with a guess $x^{1,\delta}$ and  stops at the minimum $k_*=k(\delta,V^\delta)$, such that, for a given $\tau>2$ (see \cite{kaltenbacher2008}, equation (2.14) ), 
\begin{equation}\label{equation9}
\|V^\delta-F(x^{ {k_*},\delta})\|_{L^2(0,T)}\le\tau\delta.  
\end{equation}

It is possible to show that, under certain conditions (we assume that is the case), $x^{ {k_*},\delta}$ converges to a solution of $F(x)=V$ as $\delta\to0$; see \cite{kaltenbacher2008} {Theorem 3.22}. 

\subsection{Inverse Problem to obtain conductances in the H-H model}\label{subsection2.1}
The present goal is to estimate the maximum conductances $G_{\Na}$, $G_\K$ and $G_L$ while assuming that~\eqref{equation7} holds. We assume that the exponents are $a=3$, $b=1$, and $c=4$.

We denote our unknown parameters such as $x =\bG = (G_{\Na},G_\K, G_L)$, then from iteration \eqref{equation8} we have 
\begin{equation}\label{equation012}
\bG^{ {k+1},\delta}=\bG^{ k,\delta}+w^{k,\delta}F'(\bG^{ k,\delta})^*(V^\delta-F(\bG^{ k,\delta})). 
\end{equation}

 Given an initial approximation $\bG^{1,\delta}$
 and $V^\delta$, we obtain a regularizing approximation $\bG^{k_{*},\delta}$ for $\bG$, from Landweber iteration \eqref{equation012}. We denote $\bG^{k,\delta}=(G_{\Na}^{k,\delta},G_\K^{k,\delta},G_L^{k,\delta})$.

 In the next theorem, we compute the adjoint of the Gateaux derivative $F'(\bG^{k,\delta})^*$ to optimize from~\eqref{equation012}.
\begin{Theorem}\label{Theorem1}
It follows from~\eqref{equation012} that 
\begin{equation}\label{equa26}
\left( G_{\Na}^{k+1,\delta},G_\K ^{k+1,\delta},G_L ^{k+1,\delta}\right)= \left( G_{\Na}^{k,\delta},G_\K ^{k,\delta},G_L ^{k,\delta}\right)+w^{k,\delta}\left(X_{\Na}^{k,\delta},X_K^{k,\delta},X_L^{k,\delta}\right),
\end{equation}
where 
\[
w^{k,\delta}=\frac{{\|V^\delta-V^{k,\delta}\|}^2_{L^2(0,T)}}{ {\left\|\left(X_{\Na}^{k,\delta},X_K^{k,\delta},X_L^{k,\delta}\right) \right\|}^2_{\mathbb{R}^3 } },
\] 
and
\begin{eqnarray}
X_{\Na}^{k,\delta}&=&\int_0^T  {\left(m^{k,\delta}\right)}^a {\left(h^{k,\delta}\right)}^b(V^{k,\delta}-E_{\Na})U^{k,\delta}\;dt,\label{equa14}
\\ 
X_K^{k,\delta}&=&\int_0^T {\left(n^{k,\delta}\right) }^c(V^{k,\delta}-E_\K) U^{k,\delta}\;dt,\label{equa15}
\\
X_{L}^{k,\delta}&=&\int_0^T {\left(n^{k,\delta}\right) }^c(V^{k,\delta}-E_\K) U^{k,\delta}\;dt.\label{equa16}
\end{eqnarray}
The functions $m^{k,\delta}$, $n^{k,\delta}$, $h^{k,\delta}$ and $V^{k,\delta}$ solve, given $G_{\Na}^{k,\delta}$, $G_\K^{k,\delta}$ and $G_{L}^{k,\delta}$, 
\begin{equation}\label{equati1}
\left \{\begin{array}{l}\displaystyle
C_M\dot{ V}^{k,\delta}=I_{\ext}- G_{\Na}^{k,\delta} {\left(m^{k,\delta}\right)}^a {\left(h^{k,\delta}\right)}^b(V^{k,\delta}-E_{\Na}) -G_\K^{k,\delta}{\left(n^{k,\delta}\right) }^c(V^{k,\delta}-E_\K)\vspace*{0.1cm}
\\
\hspace*{1.5cm} -G_{L}^{k,\delta}(V^{k,\delta}-E_L),\vspace*{0.2cm}
\\
\dot{\mathcal{X}} =(1-\mathcal{X})\alpha_\mathcal{X} (V^{k,\delta})-\mathcal{X}\beta_\mathcal{X}(V^{k,\delta})
\quad\text{for }\mathcal{X}=m^{k,\delta},n^{k,\delta},h^{k,\delta}, \vspace*{0.2cm}
\\
V^{k,\delta}(0)=V_0,\;\;\;m^{k,\delta}(0)=m_0,\;\;\;n^{k,\delta}(0)=n_0,\;\;\;h^{k,\delta}(0)=h_0,
\end{array} \right.
\end{equation}
and $\alpha_\mathcal{X}$, $\beta_\mathcal{X}$ are defined by~\eqref{e:alphasdef}. 
Finally, $U^{k,\delta}$ solve, given $m^{k,\delta}$, $n^{k,\delta}$, $h^{k,\delta}$ and $V^{k,\delta}$,
\begin{equation}\label{ab6.44}
\left \{\begin{array}{l}\displaystyle
C_M\dot {U}^{k,\delta}- \left( G_{\Na}^{k,\delta} {\left(m^{k,\delta}\right)}^a {\left(h^{k,\delta}\right)}^b+G_\K^{k,\delta} {\left(n^{k,\delta}\right)}^c+ G_L^{k,\delta}\right) U^{k,\delta}\vspace*{0.2cm}
\\
-[(1-m^{k,\delta})\alpha'_{m^{k,\delta}}(V^{k,\delta})- m^{k,\delta}\beta'_ {m^{k,\delta}}(V^{k,\delta})]P^{k,\delta}\vspace*{0.2cm}
\\
-[(1- n^{k,\delta})\alpha'_ {n^{k,\delta}}(V^{k,\delta})- n^{k,\delta}\beta_ {n^{k,\delta}}'(V^{k,\delta})]Q^{k,\delta}\vspace*{0.2cm}
\\
-[(1-h^{k,\delta})\alpha'_{h^{k,\delta}}(V^{k,\delta})- h^{k,\delta}\beta'_{h^{k,\delta}}(V^{k,\delta})]R^{k,\delta}={V^\delta}-V^{k,\delta},\vspace*{0.3cm}
\\
\displaystyle \dot P^{k,\delta}-[\alpha_{m^{k,\delta}}(V^{k,\delta})+\beta_{m^{k,\delta}}(V^{k,\delta})]P^{k,\delta} =-aG_{\Na}^{k,\delta}{\left(m^{k,\delta}\right)}^{a-1}{\left(h^{k,\delta}\right)}^b(V^{k,\delta}-E_{\Na})U^{k,\delta},\vspace*{0.3cm}\\
\dot Q^{k,\delta}-[\alpha_{n^{k,\delta}}(V^{k,\delta})+\beta_{n^{k,\delta}}(V^{k,\delta})]Q^{k,\delta}
=-c G_\K^{k,\delta} {\left(n^{k,\delta}\right)}^{c-1}(V^{k,\delta}-E_\K)U^{k,\delta},\vspace*{0.3cm}
\\
\dot R^{k,\delta}-[\alpha_{h^{k,\delta}}(V^{k,\delta})+\beta_{h^{k,\delta}}(V^{k,\delta})]R^{k,\delta}
=-b G_{\Na}^{k,\delta} {\left(m^{k,\delta}\right)}^a {\left(h^{k,\delta}\right)}^{b-1}(V^{k,\delta}-E_{\Na})U^{k,\delta},\vspace*{0.3cm}
\\
U^{k,\delta}(T)=0,\hspace*{0.5cm}  P^{k,\delta}(T)=0,\hspace*{0.5cm}   Q^{k,\delta}(T)=0,\hspace*{0.5cm}   R^{k,\delta}(T)=0.
\end{array} \right.
\end{equation}
As previously mentioned, we assume that the constants  $a$, $b$, $c$, $E_{\Na}$ ,$E_\K$, $E_{L}$, $C_M$, $I_{\ext}$, $m_0$, $n_0$ and $h_0$ are known data. 
\end{Theorem}
\begin{proof}See Appendix~\ref{AppendixA}. 
\end{proof}
We next describe the computational scheme. 

\begin{algorithm}[H]\label{Algorithm1}
	%\SetLine
	%\linesnumbered
	\KwData{$V^\delta$, $\delta$ and $\tau$}
	\KwResult{Compute an approximation for $\bG$ using Landweber Iteration Scheme}
	Choose $\bG^{1,\delta}$ as an initial approximation for $\bG$\;
	Compute $m^{1,\delta}$, $n^{1,\delta}$, $h^{1,\delta}$  and $V^{1,\delta}$ from ~\eqref{equati1}, replacing $\bG^{k,\delta}$ by $\bG^{1,\delta}$\;
	k=1\;
	\While{$\tau\delta\le{\|V^\delta-V^{k,\delta}\|}_{L^2(0,T)}$}{
		Compute $U^{k,\delta}$ from~\eqref{ab6.44}\;
		Compute $\bG^{k+1,\delta}$ using~\eqref{equa26}\;
		Compute $m^{k+1,\delta}$, $n^{k+1,\delta}$, $h^{k+1,\delta}$  and $V^{k+1,\delta}$ from \eqref{equati1}, replacing $\bG^{k,\delta}$ by $\bG^{k+1,\delta}$\;
		$k\gets k+1$\;
	}
	\caption{ Landweber iteration to obtain  maximal conductances  }
\end{algorithm}

\subsection{Inverse Problem to obtain exponents in the H-H model}\label{subsection2.2}
Assume again that~\eqref{equation7} holds and that $G_{\Na}$, $G_\K$ and $G_L$ are known. The goal of this subsection is to estimate the exponents $a$, $b$ and $c$. Denoting the unknown parameters by $x=\ba=(a,b,c)$ it follows from iteration~\eqref{equation8} that 
\begin{equation}\label{equation017}
\ba^{ {k+1},\delta}=\ba^{ k,\delta}+w^{k,\delta}F'(\ba^{ k,\delta})^*(V^\delta-F(\ba^{ k,\delta})). 
\end{equation}
Given an initial approximation $\ba^{1,\delta}$ and the data $V^\delta$, we obtain a regularizing approximation $\ba^{k_{*},\delta}$ for $\ba$, from the Landweber iteration \eqref{equation017}. Denote $\ba^{k,\delta}=(a^{k,\delta},b^{k,\delta},c^{k,\delta})$. 

In the next Theorem, we compute the adjoint of the Gateaux derivative $F'(\ba^{k,\delta})^*$ from~\eqref{equation017}.
\begin{Theorem}\label{Theorem2}
Consider the iteration~\eqref{equation017}. It follows then that 
\begin{equation}\label{equati.018}
\left(a^{k+1,\delta},b^{k+1,\delta},c^{k+1,\delta}\right)=\left(a^{k,\delta},b^{k,\delta},c^{k,\delta}\right)+w^{k,\delta}\left(X_a^{k,\delta},X_b^{k,\delta},X_c^{k,\delta}\right),
\end{equation}
where $w^{k,\delta}$ satisfies 
\[
w^{k,\delta}=\frac{{\|V^\delta-V^{k,\delta}\|}^2_{ L^2(0,T)}}{ {\left\|\left(X_a^{k,\delta},X_b^{k,\delta},X_c^{k,\delta}\right) \right\|}^2_{\mathbb{R}^3 } },
\] 
and
\begin{eqnarray*}
X_a^{k,\delta}&=&\int_0^T G_{\Na}(V^{k,\delta}-E_{\Na}) {\left(m^{k,\delta}\right)}^{a^{k,\delta}} {\left(h^{k,\delta}\right)}^{b^{k,\delta}}{U^{k,\delta}}\ln(m^{k,\delta})\,dt,\label{equ16}
\\ 
X_b^{k,\delta}&=&\int_0^T G_{\Na}(V^{k,\delta}-E_{\Na}){\left(m^{k,\delta}\right)}^{a^{k,\delta}} {\left(h^{k,\delta}\right)}^{b^{k,\delta}}{U^{k,\delta}}\ln(h^{k,\delta})\,dt\label{equ17},
\\
X_c^{k,\delta}&=&\int_0^T G_\K (V^{k,\delta}-E_\K) {\left(n^{k,\delta}\right)}^{c^{k,\delta}}{U^{k,\delta}}\ln( {n^{k,\delta}})\,dt.
%	\label{equ18}
\end{eqnarray*}
The functions $m^{k,\delta}$, $n^{k,\delta}$, $h^{k,\delta}$ and $V^{k,\delta}$ solve
\begin{equation}\label{equati344}
\left \{\begin{array}{l}\displaystyle
C_M\dot{ V}^{k,\delta}=I_{\ext}- G_{\Na} {\left(m^{k,\delta}\right)}^{a^{k,\delta}} {\left(h^{k,\delta}\right)}^{b^{k,\delta}}(V^{k,\delta}-E_{\Na}) -G_\K{\left(n^{k,\delta}\right) }^{c^{k,\delta}}(V^{k,\delta}-E_\K)\vspace*{0.1cm}\\
\hspace*{1.5cm} -G_{L}(V^{k,\delta}-E_L),\vspace*{0.2cm}\\
\displaystyle \dot{\mathcal{X}} =(1-\mathcal{X})\alpha_\mathcal{X} (V^{k,\delta})-\mathcal{X}\beta_\mathcal{X}(V^{k,\delta}); \hspace*{0.4cm}\mathcal{X}=m^{k,\delta},n^{k,\delta},h^{k,\delta}, \vspace*{0.2cm}\\
V^{k,\delta}(0)=V_0;\;\;\;m^{k,\delta}(0)=m_0;\;\;\;n^{k,\delta}(0)=n_0;\;\;\;h^{k,\delta}(0)=h_0,
\end{array} \right.
\end{equation}
where $a^{k,\delta}$, $b^{k,\delta}$ and $c^{k,\delta}$ are given. Also, $U^{k,\delta}$ solve
\begin{equation}\label{equati377}
\left \{\begin{array}{l}\displaystyle
C_M\dot U^{k,\delta}-\left(G_{\Na} {\left(m^{k,\delta}\right)}^{a^{k,\delta}}{\left(h^{k,\delta}\right)}^{b^{k,\delta}}+G_\K {\left(n^{k,\delta}\right)}^{c^{k,\delta}}+G_L\right) {U^{k,\delta}}\\
\hspace*{1.cm} -[(1- {m^{k,\delta}})\alpha'_{m^{k,\delta}}(V^{k,\delta})- m^{k,\delta}\beta'_{m^{k,\delta}}(V^{k,\delta})]{P^{k,\delta}}\\
\hspace*{1.cm} -[(1- n^{k,\delta})\alpha'_{n^{k,\delta}}(V^{k,\delta})- n^{k,\delta}\beta_ {n^{k,\delta}}'(V^{k,\delta})]Q^{k,\delta}\\
\hspace*{1.cm} -[(1-h^{k,\delta})\alpha'_{h^{k,\delta}}(V^{k,\delta})- h^{k,\delta}\beta'_{h^{k,\delta}}(V^{k,\delta})]R^{k,\delta}=V^\delta-V^{k,\delta},\\
\displaystyle {\dot P}^{k,\delta}-[\alpha_{m^{k,\delta}}(V^{k,\delta})+\beta_{m^{k,\delta}}(V^{k,\delta})]P^{k,\delta} =\\\hspace*{1.cm}-a^{k,\delta}G_{\Na} {\left(m^{k,\delta}\right) }^{ a^{k,\delta}-1} {\left(h^{k,\delta}\right)}^{b^{k,\delta}}(V^{k,\delta}-E_{\Na})U^{k,\delta},\vspace*{0.2cm}\\
\displaystyle {\dot Q}^{k,\delta}-[\alpha_{n^{k,\delta}}(V^{k,\delta})+\beta_{n^{k,\delta}}(V^{k,\delta})]Q^{k,\delta} =\\\hspace*{1.cm}-c^{k,\delta}G_\K {\left(n^{k,\delta}\right)}^{c^{k,\delta}-1}(V^{k,\delta}-E_\K)U^{k,\delta},\vspace*{0.2cm}\\
\displaystyle {\dot R}^{k,\delta}-[\alpha_{h^{k,\delta}}(V^{k,\delta})+\beta_{h^{k,\delta}}(V^{k,\delta})]R^{k,\delta} =\\\hspace*{1.cm}-b^{k,\delta}G_{\Na} {\left(m^{k,\delta}\right)}^{a^{k,\delta}} {\left(h^{k,\delta}\right)}^{b^{k,\delta}-1}(V^{k,\delta}-E_{\Na})U^{k,\delta},\vspace*{0.2cm}\\
U^{k,\delta}(T)=0;\hspace*{0.5cm}  P^{k,\delta}(T)=0;\hspace*{0.5cm}   R^{k,\delta}(T)=0;\hspace*{0.5cm}   Q^{k,\delta}(T)=0, 
\end{array} \right.
\end{equation}
given $m^{k,\delta}$, $n^{k,\delta}$, $h^{k,\delta}$ and $V^{k,\delta}$.  The constants  $ G_{\Na}$, $G_\K$, $E_{\Na}$ ,$E_\K$, $E_{L}$, $C_M$, $I_{\ext}$, $m_0$, $n_0$ and $h_0$ are given data. 
\end{Theorem}
\begin{proof}
See Appendix \eqref{AppendixB}. 
\end{proof}

We next describe the computational scheme.

\begin{algorithm}[H]\label{Algorithm2}
	%\SetLine
	%\linesnumbered
	\KwData{$V^\delta$, $\delta$ and $\tau$}
	\KwResult{Compute an approximation for $\ba$ using Landweber Iteration Scheme}
	Choose $\ba^{1,\delta}$ as an initial approximation for $\ba$\;
	Compute $m^{1,\delta}$, $n^{1,\delta}$, $h^{1,\delta}$  and $V^{1,\delta}$ from ~\eqref{equati344}, replacing $\ba^{k,\delta}$ by $\ba^{1,\delta}$\;
	k=1\;
	\While{$\tau\delta\le{\|V^\delta-V^{k,\delta}\|}_{L^2(0,T)}$}{
		Compute $U^{k,\delta}$ from~\eqref{equati377}\;
		Compute $\ba^{k+1,\delta}$ using~\eqref{equati.018}\;
		Compute $m^{k+1,\delta}$, $n^{k+1,\delta}$, $h^{k+1,\delta}$  and $V^{k+1,\delta}$ from~\eqref{equati344}, replacing $\ba^{k,\delta}$ by $\ba^{k+1,\delta}$\;
		$k\gets k+1$\;
	}
	\caption{ Landweber iteration to obtain  exponents.}
\end{algorithm}

\section{Numerical simulation}\label{section3}
To design our numerical experiments, we first choose $x$ ($x=\bG$ or $x=\ba$) and compute $V$ from~\eqref{equation7}. Of course, in practice, 
the  values of $V$ are given by some experimental measurements, and thus subject to experimental/measurement errors. In our examples, for a given $\delta$, the noisy $V^\delta$ is obtained from 
\begin{equation}\label{equ22}
V^\delta(t)=V(t)+V(t) {\rand}_\varepsilon (t), \;\;\;\text{for all  } t \in [0,T]
\end{equation}
where $\rand_\varepsilon$ is  a uniformly distributed random variable taking values in the range $[-\varepsilon,\varepsilon]$, and $\varepsilon=\delta/\|V\|_{L^2(0,T)}$.

Next, given the initial guess $x^{1,\delta}$ and the data $V^\delta$  and $\delta$, we start to recover $x$ using Algorithm~\ref{Algorithm1} (for $x=\bG$) or Algorithm~\ref{Algorithm2} (for $x=\ba$). Note that we have the exact $x$, and we use that to gauge the algorithm performance. 

The absolute error of $V^\delta$ and its approximation $V^{k,\delta}$ defines the residual from 
\begin{equation}\label{equti28}
\Res_{k}= \| V^{\delta}-V^{k,\delta}\|_{L^2(0,T)}=\sqrt{\int_0^L  \left( V^{\delta}(t)-V^{k,\delta}(t)\right)^2dt},\;\;\;k=1,2,\cdots,k_*.
\end{equation}
The percent error of vector $x\in \mathbb{R}^3$ is defined by
\begin{equation}\label{equ23}
\Error_{k}^x={ \frac{\|x-x^{k,\delta}\|_{\mathbb{R}^3}}{\|x\|_{ \mathbb{R}^3}}}\times 100\%,\;\;\;k=1,2,\cdots,k_*. 
\end{equation}

Each step of Algorithm~\ref{Algorithm1} and Algorithm~\ref{Algorithm2} involves solving two ODEs. Of course, there is no analytical solution for those equations, and the use of numerical methods is necessary. We use explicit Euler with a fixed time step $\Delta t$.

In this section we will present two numerical simulations. In Example~\ref{Exa2.1} we estimate the conductances $G_{\Na}$, $G_\K$ and $G_L$, and in Example~\ref{Exa3.1} we estimate the exponents $a$, $b$ and $c$. Our simulation were computed with Matlab R2012b on a Dell PC, running on a Intel(R) Core(TM) i7-4790 CPU @ 3.60GHz with 32 GB of RAM.

See the code in the URL:\url{https://github.com/MandujanoValle/Conductances-HH}, to estimate the conductances $G_{\Na}$, $G_\K$ and $G_L$, and URL:\url{https://github.com/MandujanoValle/Exponents-HH}, to estimate the exponents $a$, $b$ and $c$.

\Example \label{Exa2.1} This example is a particular case from \eqref{equation7}, with values (see~\cite{cooley1966}, page 586): 
$C_M=1\;[\mu F/cm^2]$,  $E_{\Na}=115\;[mV]$, 
$E_\K=-12\;[mV]$, $E_L=10.598\;[mV]$, $G_{\Na}=120\;[mS/cm^2]$, $G_\K=36\;[mS/cm^2]$,  $G_L=0.3\;[mS/cm^2]$,   $I_{\ext}=0\;[\mu A/cm^2]$, $a=3$, $b=1$ and $c=4$. Let  the initial conditions  $V(0)=-25\;[mV]$, $m(0)=0.5$, $n(0)=0.4$ and $h(0)=0.4$. We consider $T=10 \;[mS]$ and $\Delta t=0.02$.  Given $V^\delta$, the goal of this example is to approximate $\bG=(G_{\Na}, G_\K, G_L)\;[mS/cm^2]$. 

First, given $\bG = (120,36,0.3)\;[mS/cm^2]$, we compute $V$ from \eqref{equation7}  . Then, we calculate $V^\delta$ from \eqref{equ22} given $\varepsilon$ (see table~\ref{table1}). Next, we consider $V$ and $\bG$ as unknowns.

In this test we consider the initial guess $\bG^{1,\delta}=(0,0,0)\;[mS/cm^2]$ and $\tau=2.01$. Table~\ref{table1} presents the results for various levels of noise. When $\varepsilon$ decreases, the number of iterations grow resulting in a better approximation for $\bG=(G_{\Na},G_\K,G_L)\;[mS/cm^2]$ and smaller residuals. As expected, the result of the last column is close to $\tau\delta$, related to the stopping criteria~\eqref{equation9}. 

In Figures \ref{Figure1-1}, \ref{Figure1-2} and \ref{Figure1-3}, we plot some results for $\varepsilon=5\%$ (Table 2, line 4). 
 \begin{table}[ht!]
	{\small	\centering
		\begin{tabular}{|l|l|l|l|l|l|l|l|l|l|}\hline
$\varepsilon $  &$k_* $    &$G_{\Na}^{k_*,\delta}$ & $G_\K^{k_*,\delta}$ &$G_L^{k_*,\delta}$& $Error_{k_*}^x$            &$Res_{k_*}$ \\\hline     
$125\%$ &$1$     &$0$      &$0     $ & $0 $     & $100  \;\%$ &$161 $   \\\hline	           	
$25\%$  &$19303$ &$114.08$ &$28.49 $ & $8.1727$ & $9.9 \;\%$  &$49$    \\\hline
$5\%$   &$25012$ &$115.07$ &$30.59 $ & $0.7938$ & $5.8 \;\%$  &$10 $   \\\hline
$1\%$   &$33419$ &$119.10$ &$34.16 $ & $0.3221$ & $1.6 \;\%$  &$2$   \\\hline
$0.2\%$ &$48642$ &$119.82$ &$35.62 $ & $0.3043$ & $0.3 \;\%$  &$0.4$   \\\hline 

		\end{tabular}
		\caption{Numerical results for Example \ref{Exa2.1} for various values of  $\varepsilon$, as in \eqref{equ22}. The second column contains the number of iterations according to  \eqref{equation9}. The third, fourth and fifth columns are the approximations  for $G_{\Na}$, $G_\K$ and $G_L$ respectively. The sixth column is the relative error of $\bG=(G_{\Na},G_\K,G_L)$ according to 
        \eqref{equ23}. The last column is the residue, see   \eqref{equti28}. }
		\label{table1} } 	
\end{table}

\begin{figure}[ht!]
	\centering
	\includegraphics[height=7cm, width=12cm]{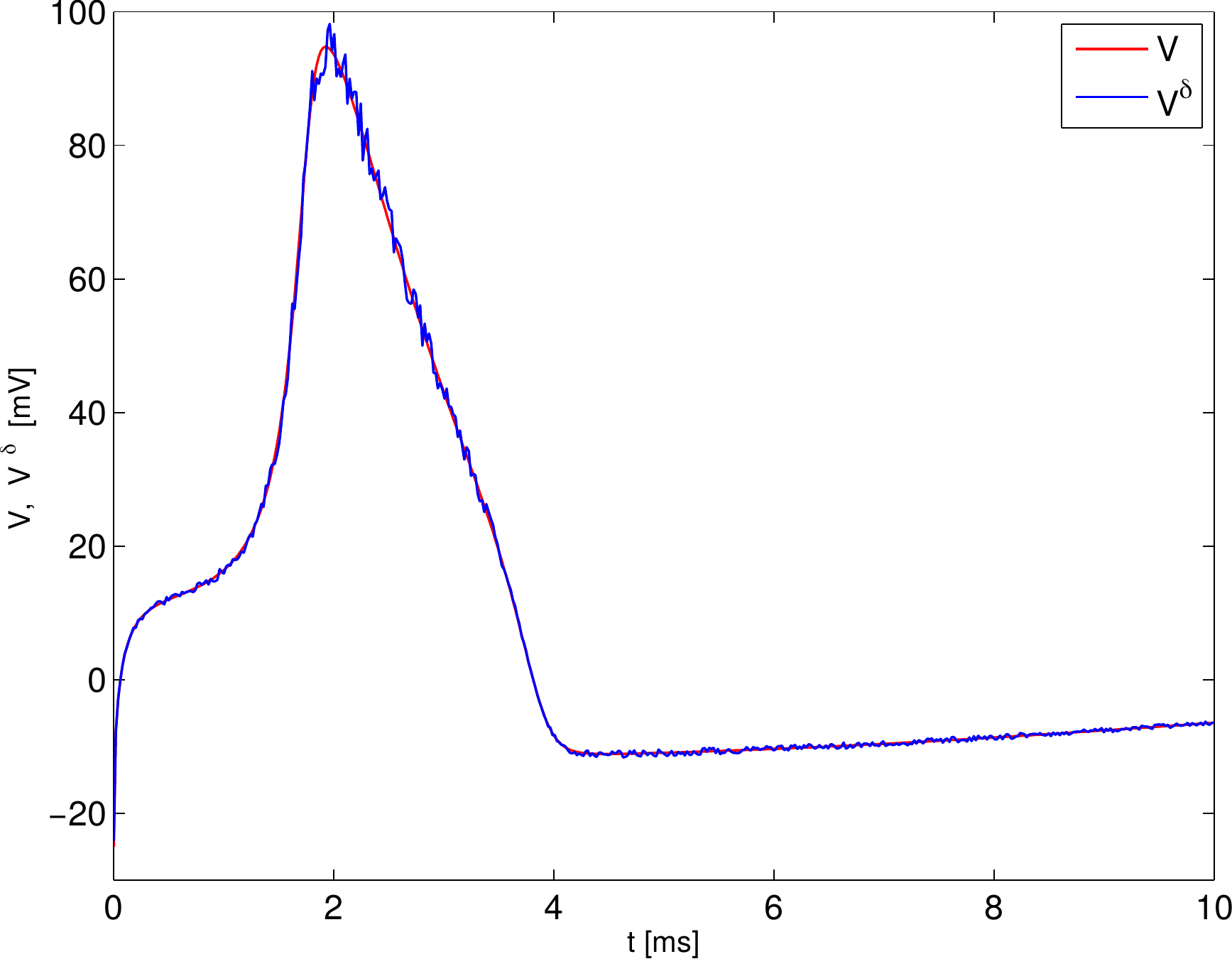}
	\caption{For Example \ref{Exa2.1}. The red line ($V$) is the exact membrane potential  and  blue line ($V^\delta$) is the membrane potential measurement; in this case $\varepsilon=5\%$. }
	\label{Figure1-1}
\end{figure}
\begin{figure}[ht!]
	\centering
	\includegraphics[height=8cm, width=13cm]{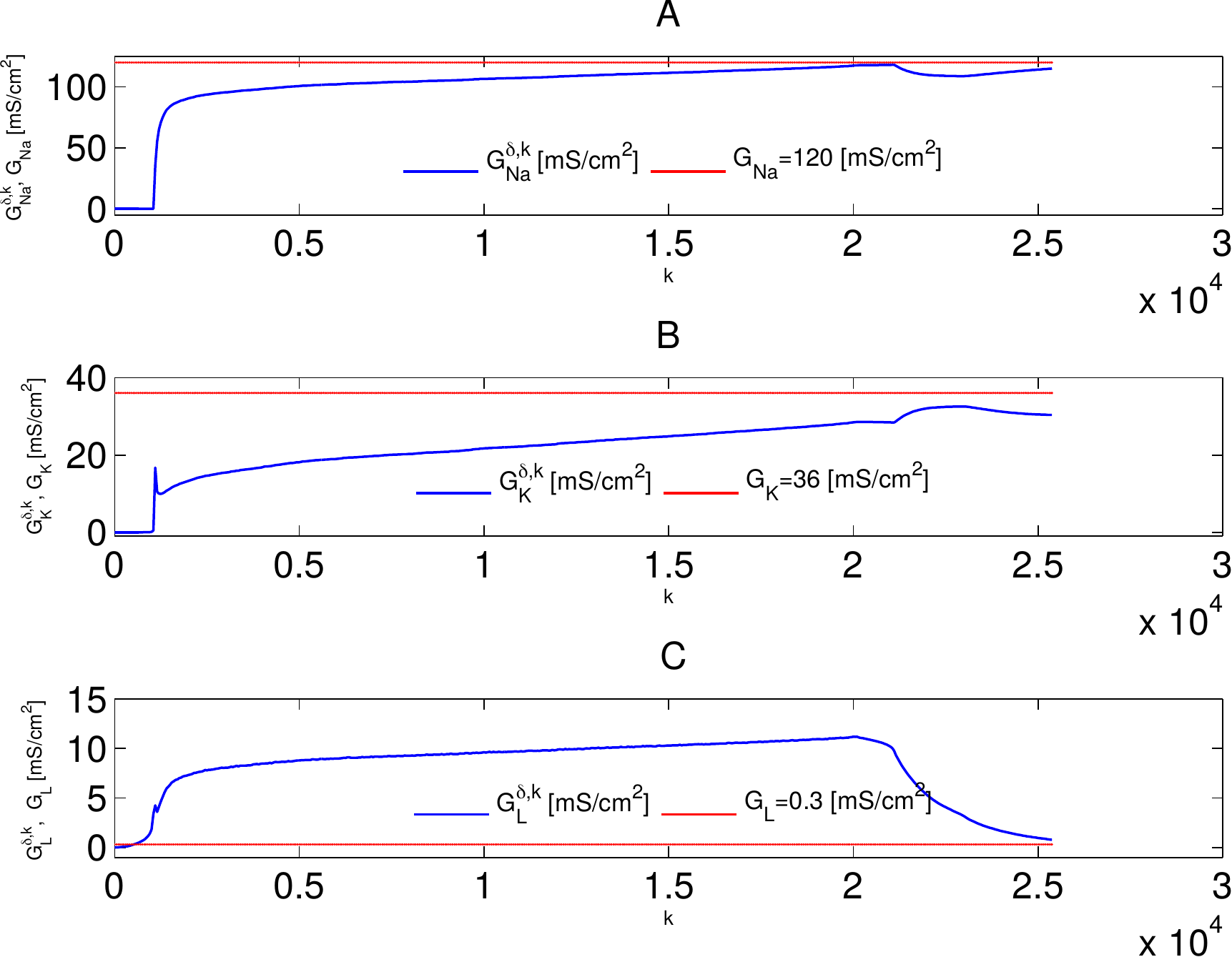}
	\caption{Figures for Example~\ref{Exa2.1} (estimation of the conductances) with $\varepsilon = 5\%$. The x-axis gives the number of iterations ($k$) and the y-axis gives the conductance.
		The red lines are the exact solutions and blue lines are the approximations.
		The figures \ref{Figure1-2}-A, \ref{Figure1-2}-B and \ref{Figure1-2}-C  display the estimates of the maximum conductances of sodium, potassium and leakage, respectively.}   
%		 $G_{\Na}$ (Subplot-A), $G_\K$ (Subplot-B) and $G_L$ (Subplot-C).}
	\label{Figure1-2}
\end{figure}

\begin{figure}[ht!]
\centering
\includegraphics[height=8cm, width=13cm]{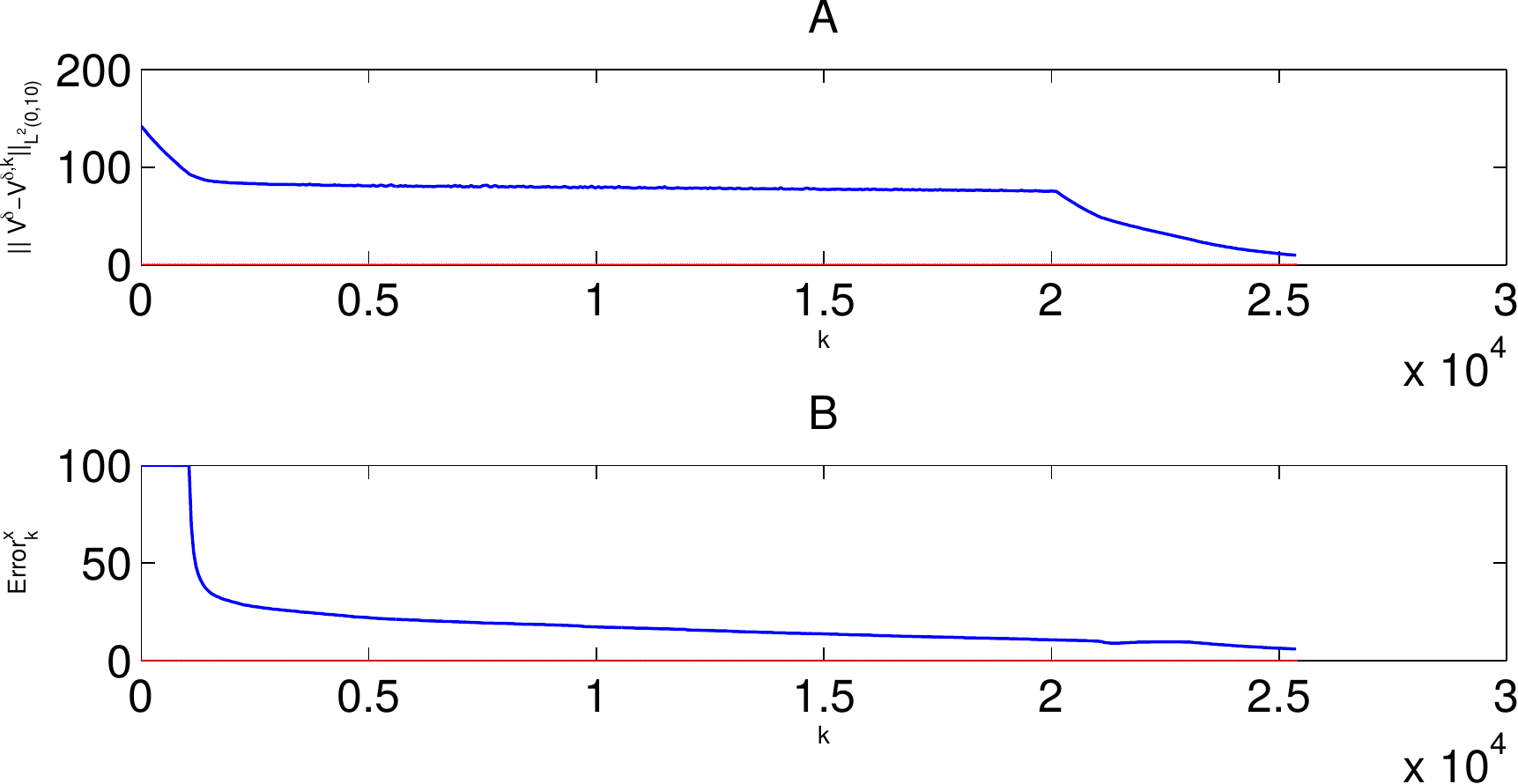}
\caption{Example~\ref{Exa2.1} with $\varepsilon =5\%$. The x-axis indicates the number of iterations ($k$). The y-axis, in the figures $A$ and $B$ are the residual \eqref{equti28} and error \eqref{equ23}, respectively.     
.
}
\label{Figure1-3}
\end{figure}

\Example \label{Exa3.1} This example is another particular case from \eqref{equation7} with values (see~\cite{cooley1966}, page 586): $C_M=1\;[\mu F/cm^2]$,  $E_{\Na}=115\;[mV]$, 
$E_\K=-12\;[mV]$, $E_L=10.598\;[mV]$, $G_{\Na}=120\;[mS/cm^2]$, $G_\K=36\;[mS/cm^2]$, $G_L=0.3\;[mS/cm^2]$
, $I_{\ext}=0\;[\mu A/cm^2]$, $a=3$, $b=1$ and $c=4$. Let the initial conditions  $V(0)=-25\;[mV]$, $m(0)=0.5$, $n(0)=0.4$ and $h(0)=0.4$. We consider the time $T=5\;[ms]$ with $\Delta t=0.02$.  Given $V^\delta$, our goal is to approximate  
$\ba=(a, b, c)=(3,1,4)$. 

First we calculate $V$  from \eqref{equation7}  given $\ba=(3,1,4)$. Then, we calculate $V^\delta$ from 
\eqref{equ22} given $\varepsilon$ (see table \ref{table1}). We then consider $V$ and $\ba$ unknown.
 
In this example we consider the initial guess $\ba^{1,\delta}=(0,0,0)$ and $\tau=2.01$. Table \ref{table3} presents the results for various levels of noise. In figures \ref{Figure2-1}, \ref{Figure2-2} and \ref{Figure2-3}, we plot some results for a level of noise $\varepsilon=1\%$. 

%125  1.00000	       0.00000	   0.00000	   0.00000	 100.00000	 169.70592	   0.00000
%25   11681.00000      1.57200	   0.49621	  -0.30097	  89.42403	  48.37718	   0.00000
%5    95605.00000	   2.97042	   0.80699	   2.62616	  27.21398	   9.67596	   0.00000
%1    188827.00000	   3.00843	   0.95395	   3.67429	   6.45334	   1.93515	   0.00000   
%0.2  283487.00000	   3.00243	   0.99016	   3.93052	   1.37702	   0.38704
  	\begin{table}[ht!]
	{\small	\centering
		\begin{tabular}{|l|l|l|l|l|l|l|l|l|l|}\hline
$\varepsilon$ & $k_*$ & $a^{ k_*,\delta}$ & $b^{ k_*,\delta}$ &$c^{ k_*,\delta}$ & $Error_{k_*}^x$  &$Res_{k_*}$ \\\hline
$125\;\%$  & $1$      & $0$      & $0$     & $0$      & $ 100\;\%$  & $170$
\\\hline
$25\;\% $  & $11681$  & $1.572$  & $0.496$ & $-0.300$ & $ 89 \;\%$  & $48$
\\\hline
$5\;\%$    & $95605$  & $2.970$  & $0.807$ & $2.626$  & $ 27\;\%$   & $9.7$
\\\hline
$1\;\%$    & $188827$ & $3.008$  & $0.954$ & $3.674$  & $  6\;\%$   & $1.9$
\\\hline
$0.2\;\%$  & $283487$ & $3.002$  & $0.990$ & $3.930$  & $  1.4\;\%$ & $0.4$\\\hline           
		\end{tabular}
		\caption{Numerical results for Example \ref{Exa3.1}. See Table \ref{table1} for a description of the contents. }
		\label{table3}}
\end{table}

\begin{figure}[ht!]
	\centering
	\includegraphics[height=7cm, width=12cm]{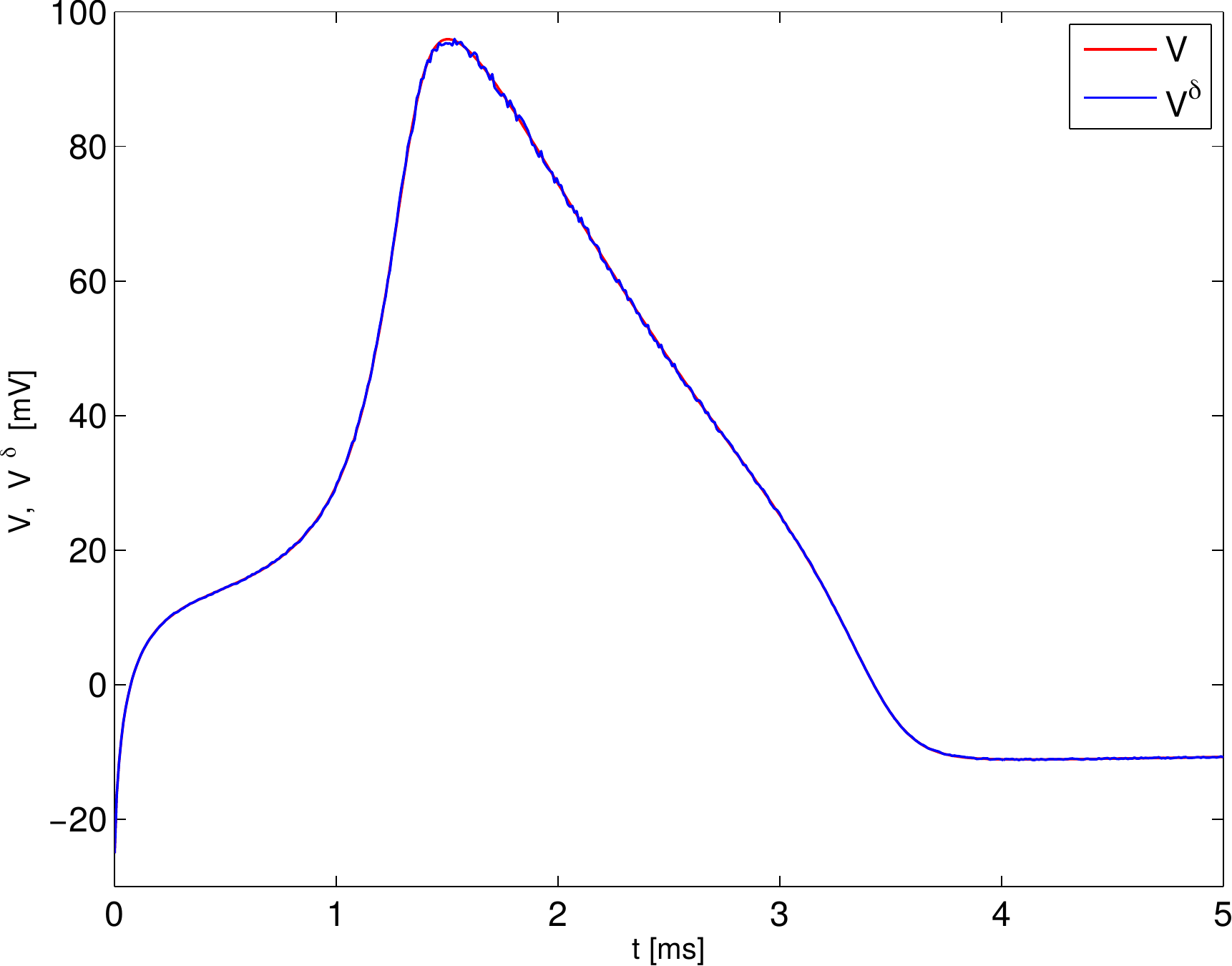}
	\caption{For Example \ref{Exa3.1} and $\varepsilon=1\%$. The red line ($V$) is the exact membrane potential  and  blue line ($V^\delta$) is the membrane potential measurement. }
	\label{Figure2-1}
\end{figure}
\begin{figure}[ht!]
	\centering
	\includegraphics[height=8cm, width=13cm]{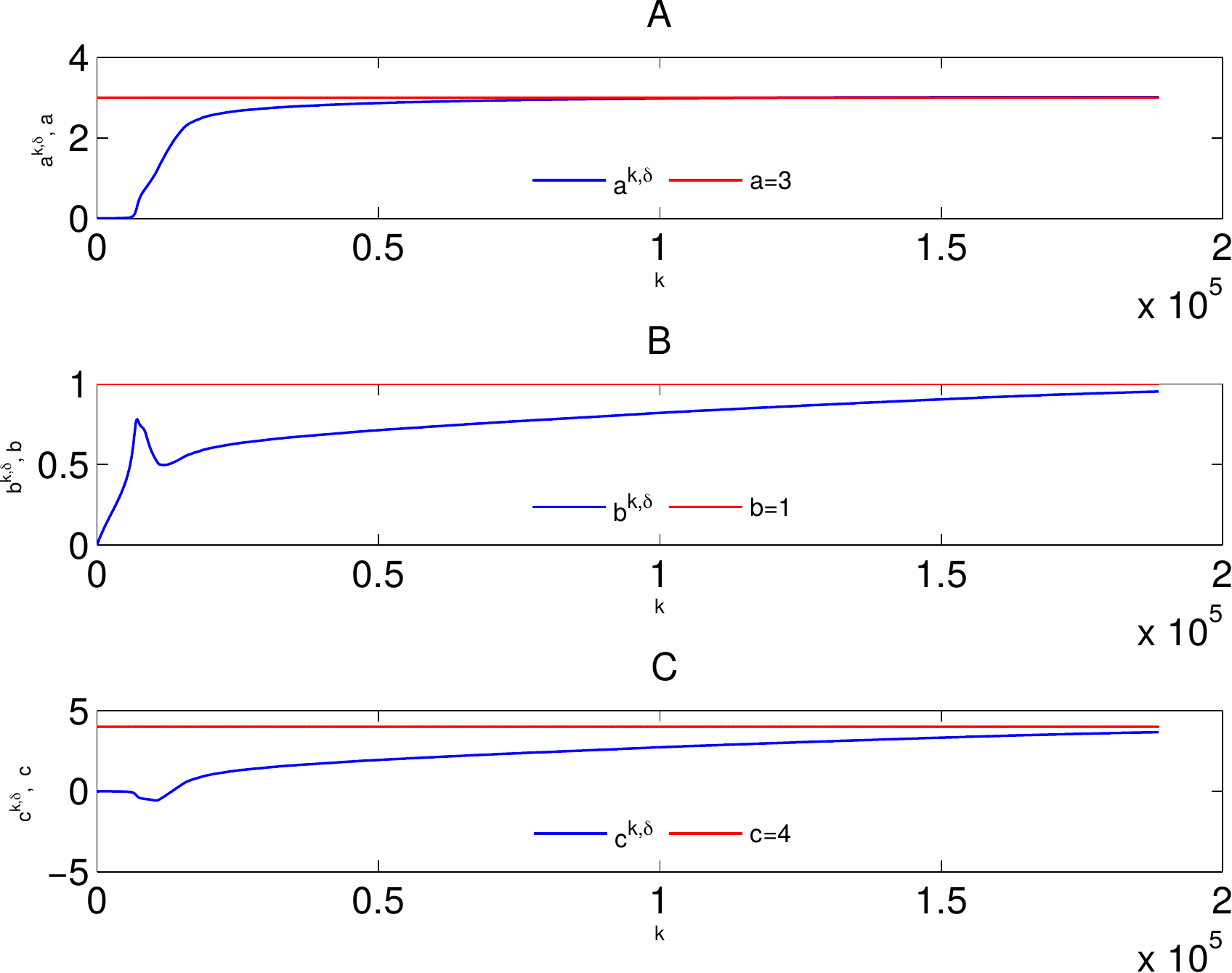}
	\caption{For Example \ref{Exa3.1} and $\varepsilon=1\%$.   The x-axis is the number of iterations ($k$). In  y-axis, the red lines are the exact solutions and blue lines are the approximations.
		The figures \ref{Figure2-2}-A, \ref{Figure2-2}-B and \ref{Figure2-2}-C  are the estimates of $a$, $b$ and $c$, respectively..}
	\label{Figure2-2}
\end{figure}

\begin{figure}[ht!]
	\centering
	\includegraphics[height=8cm, width=13cm]{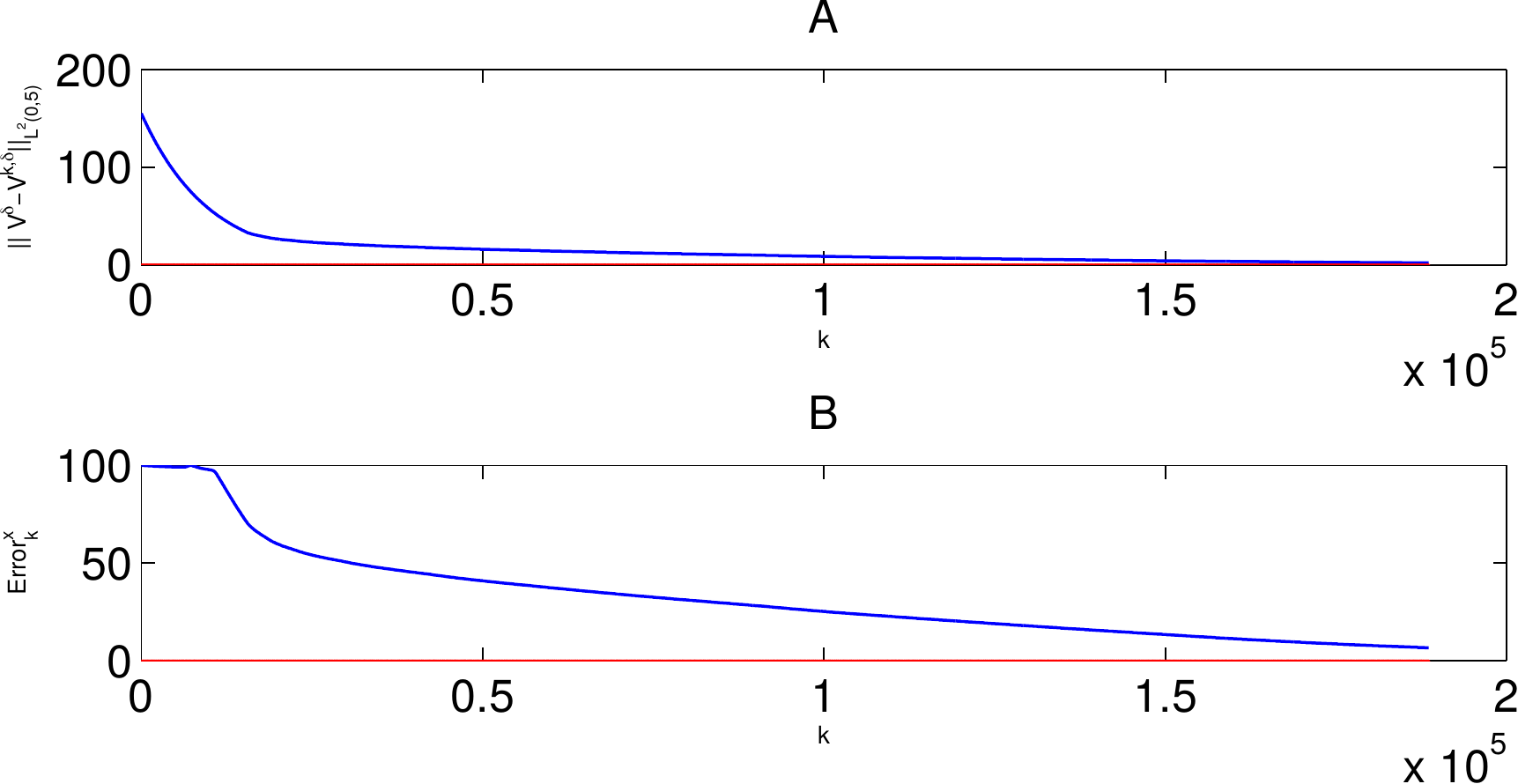}
	\caption{For Example \ref{Exa3.1} and $\varepsilon=1\%$. The x-axis is the number of iterations ($k$). The y-axis, in the figures $A$ and $B$ are the residual \eqref{equti28} and error \eqref{equ23}, respectively.}
	\label{Figure2-3}
\end{figure}

\appendix
\section{Proof of Theorem \ref{Theorem1}}\label{AppendixA}
In this Appendix, we show Theorem~\ref{Theorem1}. 
\begin{proof}
	Consider the operator F defined in \eqref{equa10}.
	Evaluating  $\bG^{k,\delta}$ in $F$ , we have $F(\bG^{k,\delta})=V^{k,\delta}$,  where $V^{k,\delta}$, $m^{k,\delta}$, $n^{k,\delta}$ and $h^{k,\delta}$ solve the ODE \eqref{equati1}.

	Let the vector  $\boldsymbol{\theta}=(\theta_{\Na},\theta_\K,\theta_L)\in \mathbb{R}^3$ and $\lambda \in \mathbb{R}$, then evaluating $\bG^{k,\delta}+\lambda\boldsymbol{\theta}$ in the operator $F$, we have 
	$F(\bG^{k,\delta}+\lambda\boldsymbol{\theta})=V^{k,\delta}_\lambda$, where $V^{k,\delta}_\lambda$, $m^{k,\delta}_\lambda$, $n^{k,\delta}_\lambda$ and $h^{k,\delta}_\lambda$ solve
	\begin{equation}\label{ab6.2}
	\left \{\begin{array}{l}\displaystyle
	C_M\dot{V}^{k,\delta}_\lambda=I_{\ext}
	-         \left(G_{\Na}^{k,\delta}+\lambda\theta_{\Na}\right){\left({m}_\lambda^{k,\delta}\right) }^a{\left({h}_\lambda^{k,\delta}\right)}^b\left(V^{k,\delta}_{\lambda}-E_{\Na}\right)\\\hspace*{0.5cm}-\left(G_\K^{k,\delta}+\lambda\theta_K\right){ \left(n_\lambda^{k,\delta}\right) }^c\left(V^{k,\delta}_{\lambda}-E_\K\right)
	-\left(G_{L}^{k,\delta}+\lambda\theta_L\right)\left(V^{k,\delta}_\lambda-E_L\right),\vspace*{0.2cm}\\
	\displaystyle \dot{\mathcal{X}} =(1-\mathcal{X})\alpha_\mathcal{X} (V^{k,\delta})-\mathcal{X}\beta_\mathcal{X}(V^{k,\delta}); \hspace*{1.4cm}\mathcal{X}=m^{k,\delta}_\lambda,n^{k,\delta}_\lambda,h^{k,\delta}_\lambda, \vspace*{0.2cm}\\
	V_\lambda^{k,\delta}(0)=V_0;\;\;\;\;m_\lambda^{k,\delta}(0)=m_0;\;\;\;\;n_\lambda^{k,\delta}(0)=n_0;\;\;\;\;n_\lambda^{k,\delta}(0)=n_0.
	\end{array} \right.
	\end{equation}
	
	The Gateaux derivative of $F$ at $\bG^{k,\delta}$ in the direction $\boldsymbol{\theta}$ is  given by 
	\begin{equation}\label{equ8}
	W^{k,\delta}=F'(\bG^{k,\delta})(\boldsymbol{\theta})=\lim_{\lambda\to0}\frac{F(\bG^{k,\delta}+\lambda\boldsymbol{\theta})-F(\bG^{k,\delta})}{\lambda}. 
	\end{equation}
	
	Also, we denote the following limits
	\begin{equation} \label{ekl13}
	M^{k,\delta}=\lim_{\lambda \rightarrow 0}\frac{m^{k,\delta}_{\lambda}-m^{k,\delta}}{\lambda},
	\hspace*{0.5cm}N^{k,\delta}=\lim_{\lambda \rightarrow 0}\frac{n^{k,\delta}_{\lambda}-n^{k,\delta}}{\lambda},
	\hspace*{0.5cm}H^{k,\delta}=\lim_{\lambda \rightarrow 0}\frac{h^{k,\delta}_{\lambda}-h^{k,\delta}}{\lambda},
	\end{equation}
	where  $M^{k,\delta}$, $N^{k,\delta}$ and $H^{k,\delta}$ are the Gateaux derivatives of $m^{k,\delta}$, $n^{k,\delta}$ and $h^{k,\delta}$, respectively. 
	
	%$\mathsf {W} $ $\boldsymbol{W}$
	
	Considering the difference between ODEs \eqref{ab6.2} and \eqref{equati1}, dividing by $\lambda$ and taking the limit $\lambda \rightarrow 0$, we have the following ODE  
	\begin{equation}\label{ab6.3}
	\left \{\begin{array}{l}\displaystyle
	C_M\dot {W} ^{k,\delta}+ \left(G_{\Na}^{k,\delta}{\left( m^{k,\delta}\right)}^a {\left( h^{k,\delta}\right)}^b+G_\K^{k,\delta} {\left( n^{k,\delta} \right)}^c+G_L^{k,\delta}\right)W^{k,\delta}=\vspace*{0.2cm}\\
	\hspace*{0.8cm}-a G_{\Na}^{k,\delta} { \left(m^{k,\delta}\right) }^{a-1} M^{k,\delta} {\left(h^{k,\delta}\right)}^b(V^{k,\delta}-E_{\Na})\vspace*{0.2cm}\\
	\hspace*{0.8cm}-bG_{\Na}^{k,\delta} {\left( m^{k,\delta}\right)}^a {\left(h^{k,\delta}\right)}^{b-1}H^{k,\delta}(V^{k,\delta}-E_{\Na})-cG_\K^{k,\delta} {\left(n^{k,\delta}\right)}^{c-1} N^{k,\delta}(V^{k,\delta}-E_\K)\vspace*{0.2cm}\\
	\hspace*{0.8cm} -\theta_{\Na} {\left(m^{k,\delta}\right)}^a{\left(h^{k,\delta}\right)}^b(V^{k,\delta}-E_{\Na})-\theta_K {\left(n^{k,\delta}\right)}^c(V^{k,\delta}-E_\K) -\theta_L(V^{k,\delta}-E_L),\vspace*{0.3cm} \\
	\displaystyle \dot{\mathcal{X}}+[\alpha_\mathcal{Y}(V^{k,\delta})+\beta_\mathcal{Y}(V^{k,\delta})]\mathcal{X} =[(1-\mathcal{Y})\alpha'_\mathcal{Y}(V^{k,\delta})-\mathcal{Y}\beta'_\mathcal{Y}(V^{k,\delta})]W^{k,\delta};\vspace*{0.2cm} \\ \hspace*{0.2cm}
	(\mathcal{X},\mathcal{Y})=( M^{k,\delta},m^{k,\delta}),(N^{k,\delta},n^{k,\delta}),(H^{k,\delta},h^{k,\delta}),\vspace*{0.3cm}\\
	W^{k,\delta}(0)=0;\hspace*{0.5cm}M^{k,\delta}(0)=0;\hspace*{0.5cm}N^{k,\delta}(0)=0;\hspace*{0.5cm}H^{k,\delta}(0)=0.
	\end{array} \right.
	\end{equation}
This last equation is yet another system of coupled nonlinear differential equations, depending on the parameter $\boldsymbol{\theta}=(\theta_{\Na},\theta_K,\theta_L)$, representing an arbitrary point in $\mathbb{R}^3$.
		
From  Landweber iteration \eqref{equation012} and $\boldsymbol{\theta} \in \mathbb{R}^3$ arbitrary, we have
		\begin{eqnarray*}
		\langle \bG^{k+1,\delta}-\bG^{k,\delta},\boldsymbol{\theta} \;\rangle_{\mathbb{R}^3}
		&=&w^{k,\delta}\langle F'(\bG^{k,\delta})^*(V^\delta-F(\bG^{k,\delta})),\boldsymbol{\theta}\;\rangle_{\mathbb{R}^3},\\
		&=&w^{k,\delta}\langle F'(\bG^{k,\delta})^*(V^\delta-V^{k,\delta}),\boldsymbol{\theta}\;\rangle_{\mathbb{R}^3}.
	\end{eqnarray*}
By  definition of adjoint operator
\[
\langle \bG^{k+1,\delta}-\bG^{k,\delta},\boldsymbol{\theta} \;\rangle_{\mathbb{R}^3}=w^{k,\delta}\langle V^\delta-V^{k,\delta},F'(x_k)(\boldsymbol{\theta})\;\rangle_{L^2[0,T]}, 
\]
where the internal product in  $L^2[0,T]$ is given by $\Phi=\int_0^T ( V^\delta-V^{k,\delta} )W^{k,\delta}\;dt$, and from~\eqref{equ8} and the previous equation,
	\[\langle \bG^{k+1,\delta}-\bG^{k,\delta},\boldsymbol{\theta} \;\rangle_{\mathbb{R}^3}=w^{k,\delta}\langle V^\delta-V^{k,\delta},W^{k,\delta}\rangle_{L^2[0,T]}.\]
Denoting the last equality by  $\Phi$, we gather that 
\begin{equation}\label{ab6.5}
	\Phi=		\frac{\langle \bG^{k+1,\delta}-\bG^{k,\delta},\boldsymbol{\theta} \;\rangle_{\mathbb{R}^3}}{w^{k,\delta}}=\langle V^\delta-V^{k,\delta},W^{k,\delta}\rangle_{L^2[0,T]}.  
	\end{equation}

From the previous equation and  the first equality from ODE \eqref{ab6.44}, we obtain 
	\begin{multline}\label{ab6.6}
	\Phi  =\int_0^T  \left( C_M\dot {U}^{k,\delta}W^{k,\delta}-(G_{\Na}^{k,\delta}{\left( m^{k,\delta}\right)}^a {\left( h^{k,\delta}\right)}^b+G_\K^{k,\delta}{\left( n^{k,\delta}\right)}^c+G_L^{k,\delta})U^{k,\delta}W^{k,\delta}\right) \,dt\\
	-\int_0^T \left[(1-{m^{k,\delta}})\alpha_{m^{k,\delta}}' (V^{k,\delta})-{m^{k,\delta}}\beta_{m^{k,\delta}}'(V^{k,\delta})\right]P^{k,\delta}W^{k,\delta}\,dt \\-\int_0^T\left[(1-{n^{k,\delta}})\alpha_{n^{k,\delta}}' (V^{k,\delta})-{n^{k,\delta}}\beta_{n^{k,\delta}}'(V^{k,\delta})\right]Q^{k,\delta}W^{k,\delta}\,dt\\
	-\int_0^T\left[(1-{h^{k,\delta}})\alpha_{h^{k,\delta}}' (V^{k,\delta})-{h^{k,\delta}}\beta_{h^{k,\delta}}'(V^{k,\delta})\right]R^{k,\delta}W^{k,\delta}\,dt.
	\end{multline}
Integrating the first term from \eqref{ab6.6} by parts, and from the initial $(W^{k,\delta}(0)=0 )$ and final $(U^{k,\delta}(T)=0)$ conditions, we obtain
	\begin{equation}\label{ab6.7}
	\int_0^T   C_M\dot {U}^{k,\delta}W^{k,\delta}= \int_0^T C_MU^{k,\delta}\dot{W}^{k,\delta}.
	\end{equation}
Replacing  equation \eqref{ab6.7} in  \eqref{ab6.6}, we have 
\begin{eqnarray*}
		\Phi &=&-\int_0^T  \left( C_M\dot{W}^{k,\delta}+(G_{\Na}^{k,\delta}{\left( m^{k,\delta}\right)}^a{\left( h^{k,\delta}\right)}^b+G_\K^{k,\delta}{\left( n^{k,\delta}\right)}^c+G_L^{k,\delta})W^{k,\delta}\right) U^{k,\delta}\;dt\\
		& &-\int_0^T \left[(1-{m^{k,\delta}})\alpha_{m^{k,\delta}}' (V^{k,\delta})-{m^{k,\delta}}\beta_{m^{k,\delta}}'(V^{k,\delta})\right]P^{k,\delta}W^{k,\delta}\;dt\\ & &-\int_0^T\left[(1-{n^{k,\delta}})\alpha_{n^{k,\delta}}' (V^{k,\delta})-{n^{k,\delta}}\beta_{n^{k,\delta}}'(V^{k,\delta})\right]{Q^{k,\delta}}W^{k,\delta}\;dt\\
		& &-\int_0^T\left[(1-{h^{k,\delta}})\alpha_{h^{k,\delta}}' (V^{k,\delta})-{h^{k,\delta}}\beta_{h^{k,\delta}}'(V^{k,\delta})\right]R^{k,\delta}W^{k,\delta}\;dt.
\end{eqnarray*}
Replacing, the first equality from the ODE~\eqref{ab6.3}, in the first integral from the previous equation, we gather 
	\begin{multline}\label{ab6.8}
	\Phi=\int_0^T  aG_{\Na}^{k,\delta}{m^{k,\delta}}^{a-1}{M^{k,\delta}}{\left( h^{k,\delta}\right)}^b(V^{k,\delta}-E_{\Na})U^{k,\delta}\;dt
	\\+\int_0^T bG_{\Na}^{k,\delta}{\left( m^{k,\delta}\right)}^a{h^{k,\delta}}^{b-1}\mathsf{H}(V^{k,\delta}-E_{\Na})U^{k,\delta}\;dt
	+\int_0^T cG_\K^{k,\delta}{n^{k,\delta}}^{c-1}\mathsf{N}(V^{k,\delta}-E_\K)U^{k,\delta}\;dt\\
	+\int_0^T {\left( m^{k,\delta}\right)}^a{\left( h^{k,\delta}\right)}^b(V^{k,\delta}-E_{\Na})\alpha U^{k,\delta}\;dt
	+\int_0^T{\left( n^{k,\delta}\right)}^c(V^{k,\delta}-E_\K)\beta U^{k,\delta}\;dt \\
	+\int_0^T (V^{k,\delta}-E_L)\gamma U^{k,\delta}\;dt 
	-\int_0^T \left[(1-{m^{k,\delta}})\alpha_{m^{k,\delta}}' (V^{k,\delta})-{m^{k,\delta}}\beta_{m^{k,\delta}}'(V^{k,\delta})\right]P^{k,\delta}W^{k,\delta}\;dt\\
	-\int_0^T\left[(1-{n^{k,\delta}})\alpha_{n^{k,\delta}}' (V^{k,\delta})-{n^{k,\delta}}\beta_{n^{k,\delta}}'(V^{k,\delta})\right]Q^{k,\delta}W^{k,\delta}\;dt
	\\-\int_0^T\left[(1-{h^{k,\delta}})\alpha_{h^{k,\delta}}' (V^{k,\delta})-{h^{k,\delta}}\beta_{h^{k,\delta}}'(V^{k,\delta})\right]R^{k,\delta}W^{k,\delta}\;dt.
	\end{multline}
Multiplying the second equation from \eqref{ab6.44} by ${M^{k,\delta}}$, and integrating in the interval $[0,T]$ it follows that
	\begin{multline*}
	\int_0^TP^{k,\delta}_t{M^{k,\delta}}-\left[ \alpha_{m^{k,\delta}} (V^{k,\delta})+\beta_{m^{k,\delta}}(V^{k,\delta})\right]P^{k,\delta}M^{k,\delta}\;dt=\\-\int_0^Ta G_{\Na}^{k,\delta}{\left(m^{k,\delta}\right)}^{a-1}{\left( h^{k,\delta}\right)}^b(V^{k,\delta}-E_{\Na})U^{k,\delta}{M^{k,\delta}}\;dt.
	\end{multline*}
Integrating by parts the first term from the previous equation, and using the initial conditions ${M^{k,\delta}}(0)=0$ and $P^{k,\delta}(0)=0$ we have 
	\begin{multline*}
	\int_0^T\left(\dot M^{k,\delta}+\left[ \alpha_{m^{k,\delta}} (V^{k,\delta})+\beta_{m^{k,\delta}}(V^{k,\delta})\right]{M^{k,\delta}}\right)P^{k,\delta}\;dt=\\\int_0^Ta G_{\Na}^{k,\delta}{\left(m^{k,\delta}\right)}^{a-1}{\left( h^{k,\delta}\right)}^b(V^{k,\delta}-E_{\Na})U^{k,\delta}{M^{k,\delta}}\;dt.
	\end{multline*}
Then, from the previous equation and the second equation from ODE \eqref{ab6.3}, for $(\mathcal{X},\mathcal{Y})=({M^{k,\delta}},{m^{k,\delta}})$, 	
\begin{multline}\label{ab6.9}
	\int_0^TaG_\K^{k,\delta}{\left(m^{k,\delta}\right)}^{a-1}{\left( h^{k,\delta}\right)}^b(V^{k,\delta}-E_{\Na})U^{k,\delta}{M^{k,\delta}}\;dt=\\
	\int_0^T\left[(1-{m^{k,\delta}})\alpha_{m^{k,\delta}}' (V^{k,\delta})-{m^{k,\delta}}\beta_{m^{k,\delta}}'(V^{k,\delta})\right]W^{k,\delta} P^{k,\delta}\;dt.
\end{multline}	
Multiplying the third equation from \eqref{ab6.44} by $N^{k,\delta}$, and integrating in the interval $[0,T]$ we gather that 
	\begin{multline*}
	\int_0^T\dot Q^{k,\delta}N^{k,\delta}-\left[ \alpha_{n^{k,\delta}} (V^{k,\delta})+\beta_{n^{k,\delta}}(V^{k,\delta})\right]Q^{k,\delta}N^{k,\delta}\;dt=\\-\int_0^TcG_\K^{k,\delta} {\left(n^{k,\delta}\right)}^{c-1}(V^{k,\delta}-E_\K)U^{k,\delta}\;dt.
	\end{multline*}
Integrating by parts the first term from previous equation, and using the initial conditions 
	$N^{k,\delta}(0)=0$ and $Q^{k,\delta}(0)=0$ we have 
	\begin{multline*}
	\int_0^T\left( \dot N^{k,\delta}+\left[ \alpha_{n^{k,\delta}} (V^{k,\delta})+\beta_{n^{k,\delta}}(V^{k,\delta})\right]{N^{k,\delta}}\right)Q^{k,\delta}\;dt=\\\int_0^TcG_\K^{k,\delta}{\left(n^{k,\delta}\right)}^{c-1}(V^{k,\delta}-E_\K)U^{k,\delta}\;dt.
	\end{multline*}
Then, from the previous equation and the second equation from ODE \eqref{ab6.3}, for $(\mathcal{X},\mathcal{Y})=({N^{k,\delta}},{n^{k,\delta}})$, we have
	\begin{multline}\label{ab6.10}
	\int_0^TcG_\K^{k,\delta} {\left(n^{k,\delta}\right)}^{c-1}(V^{k,\delta}-E_\K)U^{k,\delta}\;dt=\\
	\int_0^T\left[(1-{n^{k,\delta}})\alpha_{n^{k,\delta}}' (V^{k,\delta})-{n^{k,\delta}}\beta_{n^{k,\delta}}'(V^{k,\delta})\right]W Q^{k,\delta}\;dt.
	\end{multline}
	Multiplying the fourth equation from \eqref{ab6.44} by $H^{k,\delta}$, and integrating in the interval $[0,T]$ we gather that
	\begin{multline*}
	\int_0^T\dot R^{k,\delta}H^{k,\delta}-\left[ \alpha_{h^{k,\delta}} (V^{k,\delta})+\beta_{h^{k,\delta}}(V^{k,\delta})\right]R^{k,\delta}H^{k,\delta}\;dt=\\
	-\int_0^Tb G_{\Na}^{k,\delta}{\left( m^{k,\delta}\right)}^a{\left(h^{k,\delta}\right)}^{b-1}(V^{k,\delta}-E_{\Na})U^{k,\delta}\;dt. 
	\end{multline*}
	Integrating by parts the first term from the previous equation, and using the initial conditions $H^{k,\delta}(0)=0$ and $R^{k,\delta}(0)=0$ we have,
	\begin{multline*}
	\int_0^T\left({\dot H^{k,\delta}}+\left[ \alpha_{h^{k,\delta}}(V^{k,\delta})+\beta_{h^{k,\delta}}(V^{k,\delta})\right]{H^{k,\delta}}\right)R^{k,\delta}\;dt=\\\int_0^Tb G_{\Na}^{k,\delta}{\left( m^{k,\delta}\right)}^a{\left(h^{k,\delta}\right)}^{b-1}(V^{k,\delta}-E_{\Na})U^{k,\delta}\;dt.
	\end{multline*}
Then, from the previous equation and the second equation from ODE \eqref{ab6.3}, for $(\mathcal{X},\mathcal{Y})=({H^{k,\delta}},{h^{k,\delta}})$, we have
	
	\begin{multline}\label{ab6.11}
	\int_0^Tb G_{\Na}^{k,\delta}{\left( m^{k,\delta}\right)}^a{\left(h^{k,\delta}\right)}^{b-1}(V^{k,\delta}-E_{\Na})U^{k,\delta}\;dt=\\
	\int_0^T\left[(1-h^{k,\delta})\alpha_{h^{k,\delta}}' (V^{k,\delta})-h^{k,\delta}\beta_{h^{k,\delta}}'(V^{k,\delta})\right]W^{k,\delta} R^{k,\delta}\;dt.
	\end{multline}
Substituting equations \eqref{ab6.9}, \eqref {ab6.10}, and \eqref{ab6.11} in~\eqref{ab6.8}, we have
\begin{multline}\label{ab6.12}
\Phi=\int_0^T {\left( m^{k,\delta}\right)}^a{\left( h^{k,\delta}\right)}^b(V^{k,\delta}-E_{\Na})\theta_{\Na} U^{k,\delta}\;dt
+\int_0^T{\left( n^{k,\delta}\right)}^c(V^{k,\delta}-E_\K)\theta_\K U^{k,\delta}\;dt\\ 
+\int_0^T (V^{k,\delta}-E_L)\theta_{L} U^{k,\delta}\;dt.
\end{multline}
Substituting equations \eqref{equa14}, \eqref{equa15} and \eqref{equa16} in equation \eqref{ab6.12} we gather that 
	\begin{equation}\label{ab6.13}
	\Phi=X_{\Na}^{k,\delta}\;\theta_{\Na}+X_K^{k,\delta}\;\theta_\K+X_L^{k,\delta}\;\theta_{L}=\left\langle  \left( X_{\Na}^{k,\delta},X_K^{k,\delta},X_L^{k,\delta}\right),\left(\theta_{\Na},\theta_\K,\theta_{L}\right)\right\rangle_{\mathbb{R}^3}.
	\end{equation}
From~\eqref{ab6.5} and \eqref{ab6.13} 
	$$		\frac{\langle \bG^{k+1,\delta}-\bG^{k,\delta},\boldsymbol{\theta} \;\rangle_{\mathbb{R}^3}}{w^{k,\delta}} =\left\langle  \left( X_{\Na}^{k,\delta},X_K^{k,\delta},X_L^{k,\delta}\right),\boldsymbol{\theta}\right\rangle_{\mathbb{R}^3}.$$
Since $\boldsymbol{\theta} \in \mathbb{R}^3$ is arbitrary, we obtain~\eqref{equa26}. 
\end{proof}

\section{Proof of Theorem \ref{Theorem2}}\label{AppendixB}
In what follows we prove Theorem~\ref{Theorem2}. 
	\begin{proof}
			Consider the operator F defined in \eqref{equa10}. Evaluating  $\ba^{k,\delta}$ in  $F$, we have $F(\ba^{k,\delta})=V^{k,\delta}$,  where $V^{k,\delta}$, $m^{k,\delta}$, $n^{k,\delta}$ and $h^{k,\delta}$ solve ODE \eqref{equati344}. Let the $\boldsymbol{\theta}=(\theta_{a},\theta_{b},\theta_{c})\in \mathbb{R}^3$ and $\lambda \in \mathbb{R}$, then $F(\ba^{k,\delta}+\lambda\boldsymbol{\theta})=V^{k,\delta}_\lambda$, where $V^{k,\delta}_\lambda$, $m^{k,\delta}_\lambda$, $n^{k,\delta}_\lambda$ and $h^{k,\delta}_\lambda$ solve
\begin{equation}\label{equati35}
\left \{\begin{array}{l}
C_M\dot{V}^{k,\delta}_\lambda=I_{\ext}-G_{\Na}{\left({m}_\lambda^{k,\delta}\right) }^{a^{k,\delta}+\lambda\theta_a}{\left({h}_\lambda^{k,\delta}\right)}^{b^{k,\delta}+\lambda\theta_b}\left(V^{k,\delta}_{\lambda}-E_{\Na}\right)
\\
\hspace*{0.5cm}-G_\K^{k,\delta}{\left(n_\lambda^{k,\delta}\right)}^{c^{k,\delta}+\lambda\theta_c}\left(V^{k,\delta}_{\lambda}-E_\K\right)-G_{L}\left(V^{k,\delta}_\lambda-E_L\right),\vspace*{0.2cm}
\\
\dot{\mathcal{X}} =(1-\mathcal{X})\alpha_\mathcal{X} (V^{k,\delta})-\mathcal{X}\beta_\mathcal{X}(V^{k,\delta}), \quad\text{for }\mathcal{X}=m^{k,\delta}_\lambda,n^{k,\delta}_\lambda,h^{k,\delta}_\lambda, \vspace*{0.2cm}
\\		V_\lambda^{k,\delta}(0)=V_0,\;\;\;\;m_\lambda^{k,\delta}(0)=m_0,\;\;\;\;n_\lambda^{k,\delta}(0)=n_0,\;\;\;\;n_\lambda^{k,\delta}(0)=n_0.
		\end{array} \right.
		\end{equation}
Considering the difference between the ODEs~\eqref{equati35} and~\eqref{equati344}, dividing by $\lambda$ and taking the limit $\lambda \rightarrow 0$, we have the ODE  		
		\begin{equation}\label{equati36}
		\left \{\begin{array}{l}\displaystyle
		C_M{\dot {W}}^{k,\delta}+\left(G_{\Na}{\left(m^{k,\delta}\right)}^{a^{k,\delta}} {\left(h^{k,\delta}\right)}^{b^{k,\delta}}+G_\K{\left(n^{k,\delta}\right)}^{c^{k,\delta}}+G_L\right)W^{k,\delta}=\\
		\hspace*{0.5cm}-a^{k,\delta} G_{\Na} {\left(m^{k,\delta}\right)}^{a^{k,\delta}-1}M^{k,\delta}{\left(h^{k,\delta}\right)}^{b^{k,\delta}}(V^{k,\delta}-E_{\Na})\\
		\hspace*{0.5cm}-\mathsf{b}G_{\Na} {\left(m^{k,\delta}\right)}^{a^{k,\delta}} {\left(h^{k,\delta}\right)}^{b^{k,\delta}-1}H^{k,\delta}(V^{k,\delta}-E_{\Na})\\
		\hspace*{0.5cm} -c^{k,\delta}G_\K {\left(n^{k,\delta}\right)}^{c^{k,\delta}-1}N^{k,\delta}(V^{k,\delta}-E_\K)\\
		\hspace*{0.5cm} -G_{\Na} {\left(m^{k,\delta}\right)}^{a^{k,\delta}}\ln( m^{k,\delta}) {\left(h^{k,\delta}\right)}^{b^{k,\delta}}(V^{k,\delta}-E_{\Na})\theta_a\\
		\hspace*{0.5cm}-G_{\Na}{\left(m^{k,\delta}\right)}^{a^{k,\delta}} {\left(h^{k,\delta}\right)}^{b^{k,\delta}}\ln( h^{k,\delta})(V^{k,\delta}-E_{\Na})\theta_b\\
		\hspace*{0.5cm}  -G_k{\left(n^{k,\delta}\right)}^\mathsf{c}\ln(n^{k,\delta})(V^{k,\delta}-E_\K)\theta_c,\vspace*{0.2cm}
\\
\dot{\mathcal{X}}+[\alpha_\mathcal{Y}(V^{k,\delta})+\beta_\mathcal{Y}(V^{k,\delta})]\mathcal{X} =[(1-\mathcal{Y})\alpha'_\mathcal{Y}(V^{k,\delta})-\mathcal{Y}\beta'_\mathcal{Y}(V^{k,\delta})]{W^{k,\delta}},
\\ \hspace*{0.2cm}(\mathcal{X},\mathcal{Y})=({M^{k,\delta}},m^{k,\delta}),(N^{k,\delta},n^{k,\delta}),(H^{k,\delta},h^{k,\delta}),\vspace*{0.2cm}\\
		W^{k,\delta}(0)=0,\;\;\;M^{k,\delta}(0)=0,\;\;\;N^{k,\delta}(0)=0,\;\;\;H^{k,\delta}(0)=0.
		\end{array} \right.
		\end{equation}
where $W^{k,\delta}$ is defined in equation \eqref{equ8} by replacing $\bG^{k,\delta}$ by $\ba^{k,\delta}$. Also, $M^{k,\delta}$, $N^{k,\delta}$ and $H^{k,\delta}$ are defined in equation \eqref{ekl13}. 

This last equation is again a system of coupled nonlinear differential equations, parametrized by  $\boldsymbol{\theta}=(\theta_a,\theta_b,\theta_c)$, where $\boldsymbol{\theta}\in\mathbb{R}^3$ is arbitrary. Considering~\eqref{equati377}, and proceeding as in Appendix~\ref{AppendixA}, we gather~\eqref{equati.018}. 
\end{proof}

\bibliographystyle{acm}
\bibliography{HHinverse}

\end{document}